\documentclass{amsart}
\usepackage{graphicx,fancybox}

\newtheorem{theorem}{Theorem}[section]

\newtheorem{lemma}[theorem]{Lemma}
\newtheorem{conjecture}[theorem]{Conjecture}
\newtheorem{proposition}[theorem]{Proposition}

\theoremstyle{definition}

\theoremstyle{remark}
\newtheorem{remark}[theorem]{Remark}
\numberwithin{equation}{section}

\newcommand{\Var}{\rm{Var}}

\usepackage{dsfont}
\newcommand{\RR}{\mathds{R}}

\newcommand{\E}{\mathbb{E}}

\begin{document}

\title{On integrability of quadratic harnesses}%
\author{Wlodek Bryc}%
\address{
Department of Mathematical Sciences,
University of Cincinnati,
PO Box 210025,
Cincinnati, OH 45221--0025, USA}
\email{brycw@math.uc.edu}%


\date{File: \jobname.tex Printed \today}%
\begin{abstract}    We investigate integrability properties of processes with linear regressions and quadratic conditional variances.  We establish the right order of dependence of which moments are finite  on the parameter $\sigma\tau$ defined below, raising the question of determining the optimal constant.
\end{abstract}
\maketitle
\section{Introduction and main results}

 The study of processes with linear regression and quadratic conditional variances was initiated by Pluci\'nska \cite{Plucinska:1983} who gave a characterization  of the Wiener process in terms of conditional means and conditional variances.
 This approach was extended by several authors in many directions, sometimes with a trade-off between the amount of conditioning assumed and how much information is available about the covariance.
 Of the early generalizations, we mention here the work of Szab\l owski \cite{Szabowski:1989} who extracts information from minimal assumptions on conditioning, provided that the covariance is smooth. In his influential work  Weso\l owski \cite{Wesolowski:1993} characterized processes whose conditional variances is an arbitrary quadratic polynomial in the increments of the process; besides the Wiener process, such processes turned out to be either the Poisson process, or the negative binomial process, or the gamma process, or the hyperbolic secant process.

  Processes with conditional variances given by more general quadratic polynomials have been analyzed in a series of papers that started with \cite{Bryc98}, and culminated with Askey-Wilson polynomials in \cite{Bryc-Wesolowski-08}. In these papers, the authors used algebraic techniques  to study associated  orthogonal polynomials, and to identify and construct  a number of more exotic Markov processes with linear regressions and quadratic conditional variances.

  The technique of orthogonal polynomials  relies on apriori information about the existence of  moments of all orders.
 Integrability can be deduced from assumptions on conditional variances, see \cite{Bryc85b}, \cite{Wesolowski:1993}, and more recently \cite[Theorem 2.5]{Bryc-Matysiak-Wesolowski-04}.  Unfortunately, available integrability results do not cover all possible quadratic  conditional variances, and in fact they do not use the full power of the "quadratic harness condition" \eqref{LR} and \eqref{Var} at all. Instead, they   are based on the following weaker assumptions which appeared in  \cite[(2.7),(2.8)(2.27), (2.28)]{Bryc-Matysiak-Wesolowski-04}:
%
%
%
 \begin{equation}\label{cov}
 \E(X_t)=0, \E(X_t,X_s)=\min\{s,t\}
 \end{equation}
%
 For $s<t$, the one-sided conditional moments satisfy:
\begin{equation}\label{mart}
\E(X_t|X_s)=X_s,  \; \E(X_s|X_t)=\frac{s}{t}X_t\,,
\end{equation}
and there are constants $\eta,\theta\in\RR$ and $\sigma,\tau\in[0,\infty)$ such that
 \begin{equation}\label{qmart1}
 \Var(X_t|X_s)\leq \frac{t-s}{1+\sigma s}\left(1+\eta X_s+\sigma X_s^2\right),\;
\end{equation}
\begin{equation}\label{qmart2}
 \Var(X_s|X_t)\leq \frac{s(t-s)}{t+\tau }\left(1+\theta X_t/t+\tau X_t^2/t^2\right). \end{equation}
 Our goal is to show that these expressions imply finiteness of moments of order  $1/\sqrt{\sigma\tau}$, improving a bound in  \cite[Theorem 2.5]{Bryc-Matysiak-Wesolowski-04} who prove finiteness of moments of a logarithmic order  $-\log(\sigma \tau)$.
 \begin{theorem}\label{T1}
 If $(X_t)_{t<0}$ is a square-integrable process such that \eqref{cov}, \eqref{mart},   \eqref{qmart1}, and \eqref{qmart2}
 hold for all $s<t<u$ with some $\sigma\tau> 0$ , then $\E(|X_t|^p)<\infty$ for all $0\leq p\leq \frac{1}{240\sqrt{\sigma\tau}}$.
 \end{theorem}
 \begin{remark}  From \cite{Maja:2009} it follows that for any $\sigma=\tau>0$  moments  of order $2+\frac{1}{\sqrt{\sigma\tau}}$ may fail to exist.
For comparison, Theorem \ref{T1}  implies integrability of order $p\leq 1+\frac{1}{480\sqrt{\sigma\tau}}$.
(Indeed, $1+\frac{1}{480\sqrt{\sigma\tau}}=\frac{1}{2}\left(2+\frac{1}{240\sqrt{\sigma\tau}}\right)\leq\max\left\{2,\frac{1}{240\sqrt{\sigma\tau}}\right\}$.)
\end{remark}
 \begin{remark}Of course, if $\sigma\tau=0$ then moments of any order $p>0$ are finite.  Since this is already known from  \cite[Theorem 2.5]{Bryc-Matysiak-Wesolowski-04}, here we concentrate on the case  $\sigma\tau>0$.
 \end{remark}

\section{Proof}
The proof is based on tail estimates introduced in  \cite{Bryc85b}; the main improvement over   \cite[Theorem 2.5]{Bryc-Matysiak-Wesolowski-04} comes from a more careful choice of random variables $X_{1/u}, X_u$ with $u$ close to $1$. Such a choice was suggested by \cite[Theorem 2]{Szabowski:1986}.

We shall deduce Theorem \ref{T1} from the following result.
 \begin{proposition}\label{P1} Under the assumptions of Theorem \ref{T1}, if
  $2< p+1\leq \frac{1}{240\sqrt{\sigma\tau}}$ and
 $\E(|X_{t}|^{p})<\infty$  for all $t$, then there exists $t_0>0$ such that $\E(|X_{t_0}|^{p+1})<\infty$.
 \end{proposition}
 We first show how Theorem \ref{T1} follows, and then we will prove Proposition \ref{P1}.
 \begin{proof}[Proof of Theorem \ref{T1}] It suffices to show that the moment of order $p_0=\frac{1}{240\sqrt{\sigma\tau}}$ exists. Without loss of generality we may assume that $p_0>2$, that is we consider only $\sigma\tau$ small enough.

 We apply Proposition \ref{P1} recursively. Define $p_k=p_0-k$, $k=0,1,\dots, K$, where $K=\min\{k: p_0-k\leq 2\}$.  Then $K\geq 1$ so $2<p_K+1\leq p_0 \leq \frac{1}{240\sqrt{\sigma\tau}}$ and  by square-integrability of the process, $\E(|X_t|^{p_K})<\infty$ for all $t$. So by Proposition \ref{P1}, there exists $t_0>0$ such that $\E(|X_{t_0}|^{p_{K-1}})<\infty$. However,   \eqref{mart} then implies that $\E(|X_{t}|^{p_{K-1}})<\infty$ for all $t$.
If $K=1$, the recursion ends; otherwise, if $K\geq 2$,  the recursion  can be continued: by Proposition \ref{P1},  there exists $t_0>0$ such that $\E(|X_{t_0}|^{p_{K-2}})<\infty$, and so on.  The recursion ends with $\E(|X_{t}|^{p_{1}+1})=\E(|X_{t}|^{p_{0}})<\infty$.
\end{proof}
 \subsection{Proof of Proposition \ref{P1}}
 We first consider a pair of square-integrable random variables $X,Y$ such that
  \begin{eqnarray}
 \E((X-\rho Y)^2|Y)&\leq&A+B|Y|+ (1-\rho)  \rho \delta Y^2 \label{C1}\\
  \E((Y-\rho X)^2|X)&\leq&A+B|X|+ (1-\rho)  \rho \delta X^2\label{C2}
 \end{eqnarray}
for some constants $A,B\geq 0$.
 For small enough $\delta$ we have the following tail estimate.
 \begin{lemma}\label{L1}
 Suppose $X,Y$ are square-integrable and satisfy \eqref{C1} and \eqref{C2} with some $1/2<\rho<1$, and $\delta<(1-\rho)/64$. Denote $N(t)=\Pr(|X|>t)+\Pr(|Y|>t)$ and let $K=2/\rho-1$. Then there are constants $C_1,C_2$ such that for $t>0$,
  \begin{equation}\label{N(Kt)}
 N(Kt)\leq \left(\frac{C_1}{t^2}+\frac{C_2}{t}+\frac{8\delta}{1-\rho}\right)N(t).
 \end{equation}
 \end{lemma}
 \begin{proof} Throughout the proof, $C_1,C_2$ denote generic positive constants that may change from line to line.
 As in the proof of \cite[Theorem 2.5]{Bryc-Matysiak-Wesolowski-04}, we write
 \begin{equation}\label{P1+P2}
 N(Kt)\leq 2 P_1(t)+P_2(t)+P_3(t),
 \end{equation}
where
$
P_1(t)=\Pr(|X|>t, |Y|>t)$, $P_2(t)=\Pr(|X|>Kt, |Y|\leq t)$, $P_3(t)=\Pr(|X|\leq t, |Y|>Kt)
$.
Next we observe that
\begin{multline*}
P_1(t)\leq \Pr\left(|X-\rho Y|>\frac{(1-\rho)}{\sqrt{2}}\sqrt{Y^2+t^2},|Y|>t\right)\\+ \Pr\left(|Y-\rho X|>\frac{(1-\rho)}{\sqrt{2}}\sqrt{X^2+t^2},|X|>t\right).
\end{multline*}
By \eqref{C1},  conditional Chebyschev's inequality gives
\begin{multline*}
 \Pr\left(|X-\rho Y|>\frac{(1-\rho)}{\sqrt{2}}\sqrt{Y^2+t^2},|Y|>t\right)
 \\
 \leq
2\int _{|Y|>t}\frac{A+B|Y|+(1-\rho)\rho\delta Y^2}{(1-\rho)^2(Y^2+t^2)}dP
\\
\leq  \left(\frac{A}{(1-\rho)^2 t^2}+ \frac{B}{(1-\rho)^2 t}+\frac{2\delta}{1-\rho} \right)\Pr(|Y|>t).
\end{multline*}
(Here we also used a trivial bound $\rho\leq 1$.)
Since a similar bound follows from \eqref{C2}, we get
\begin{equation}\label{Bound_P1}
P_1(t)\leq \left(\frac{C_1}{t^2}+\frac{C_2}{t}+\frac{2\delta}{1-\rho}\right)N(t),
\end{equation}
Next, we use the trivial bound $|Y-\rho X|\geq \rho |X|-|Y|$ to estimate
\begin{multline*}
P_2(t) 
 \leq  \Pr(|X|>Kt, |Y-\rho X|>\rho|X|-t)
\\
\leq  \Pr\left(|X|>Kt, |Y-\rho X|>\sqrt{2\delta\rho (1-\rho)}|X|+(\rho-\sqrt{2\delta \rho(1-\rho)})|X|-t\right)\\
\leq  \Pr\left(|X|>Kt, |Y-\rho X|>\sqrt{2\delta\rho (1-\rho)}|X|+(K (\rho-\sqrt{2\delta \rho(1-\rho)})-1)t\right).
\end{multline*}
In the last bound, we already used the inequality $\rho-\sqrt{2\delta \rho(1-\rho)}>0$, which we now further  improve. Since  $\rho>1/2$ and $\delta<(1-\rho)/64$, we get $2\rho \delta<\rho (1-\rho)/2^5<\rho^4(1-\rho)/(2-\rho)^2$, and
$$
K (\rho-\sqrt{2\delta \rho (1-\rho)})>K(\rho-(1-\rho)\rho^2/(2-\rho))=(2-\rho)\left(1- \frac{\rho(1-\rho)}{2-\rho}\right)
=1+(1-\rho)^2.$$
Therefore,
$$P_2(t)\leq  \Pr\left(|X|>Kt, |Y-\rho X|>\sqrt{2\delta\rho (1-\rho)}|X|+(1-\rho)^2t\right),
$$
and from \eqref{C2} we get
$$
P_2(t)\leq \int_{|X|>Kt} \frac{A +B|X| + (1-\rho)\rho\delta X^2}{(\sqrt{2\delta\rho (1-\rho)}|X|+(1-\rho)^2t)^2}dP.
$$
Thus
\begin{equation*}\label{Bound_P2}
P_2(t)\leq \left(\frac{C_1}{t^2}+\frac{C_2}{t}\right)N(t)+ \frac12 \Pr(|X|>Kt),
\end{equation*}
and with similar bound applied to $P_3(t)$ we get
\begin{equation}\label{P23}
P_2(t)+P_3(t)\leq \left(\frac{C_1}{t^2}+\frac{C_2}{t}\right)N(t)+ \frac12 N(Kt).
\end{equation}
Combining \eqref{P1+P2} with \eqref{Bound_P1} and \eqref{P23} we get
$$
N(Kt)\leq  \left(\frac{C_1}{t^2}+\frac{C_2}{t}+\frac{4\delta}{1-\rho}\right)N(t)+ \frac12 N(Kt).
$$
Thus \eqref{N(Kt)} follows.
 \end{proof}
Next, we prove the integrability lemma.
\begin{lemma}\label{L2} Fix $p>1$ and a pair of square-integrable random variables $X,Y$ such that $\E|X|^p+\E|Y|^p<\infty$.
Suppose $X,Y$  satisfy \eqref{C1} and \eqref{C2} with $\rho=1-1/(p+1)<1$, and $120 \delta(p+1)<1$.
Then $\E|X|^{p+1}+\E|Y|^{p+1}<\infty$.
\end{lemma}
\begin{proof}
Using the notation of Lemma \ref{L1}, we want to show that if $M_1=\int_0^\infty t^{p-2} N(t) dt<\infty$  and
$M_2=\int_0^\infty t^{p-1} N(t) dt<\infty$
then
 \begin{equation}\label{MM}
 \sup_{M>0}\int_0^M t^p N(t) dt<\infty.
 \end{equation}
Note that since $p>1$, $M_1<\infty$ is equivalent to integrability of $|X|^{p-1}+|Y|^{p-1}$. The latter follows by H\"older inequality from the assumed  integrability of $|X|^{p}+|Y|^{p}$.

 Since  $p>1$ implies that $1/2<\rho<1$ and by assumption we have $64\delta(1-\rho)=64\delta (p+1)<1$,  from Lemma \ref{L1} we see that \eqref{N(Kt)} holds.
By a change of variable from \eqref{N(Kt)} we get
\begin{multline}\label{finB}
\int_0^M t^p N(t) dt= K^{p+1}\int_0^{M/K}t^pN(Kt) dt
\\\leq
C_1K^{p+1}\int_0^{M/K}t^{p-2} N(t)dt+C_2K^{p+1}\int_0^{M/K}t^{p-1} N(t)dt\\+
\frac{8\delta K^{p+1}}{1-\rho}\int_0^{M/K}t^{p} N(t)dt
\leq
 C_1K^{p+1}M_1+C_2K^{p+1}M_2+\frac{8\delta K^{p+1}}{1-\rho}\int_0^{M}t^{p} N(t)dt.
\end{multline}
By our choice of $\rho=1-1/(p+1)$, we have  $(1-1/(p+1))^{p+1}>(1-1/(p+1))1/e>1/(2e)$, so
 $$K^{p+1}=\left(\frac{2-\rho}{\rho}\right)^{p+1}=\frac{\left(1+\frac{1}{p+1}\right)^{p+1}}{\left(1-\frac{1}{p+1}\right)^{p+1}}<2e^2.$$
Thus $\frac{8\delta K^{p+1}}{1-\rho}<16 e^2 \delta (p+1)<120 \delta (p+1)<1$, and \eqref{finB}
 implies that \eqref{MM} holds.
\end{proof}
\begin{proof}[Proof of Proposition \ref{P1}]

 Given $p>1$, choose $\rho=1-1/{(p+1)}$.   Note that $\rho>1/2$.
 Take
 $$s=\frac{\rho \sqrt{\tau}}{\sqrt{\sigma}},\; t=\frac{\sqrt{\tau}}{\rho\sqrt{\sigma}}.
 $$ Then from \eqref{cov} we get $\E(XY)=\rho$.  From \eqref{qmart1}  and \eqref{qmart2} we get

 \begin{eqnarray*}
 \E((X-\rho Y)^2|Y)&\leq&\frac{1-\rho ^2}{1+\sqrt{\sigma  \tau } \rho}\left(1+\theta  \sqrt{\rho } \sqrt[4]{\frac{\sigma }{\tau }}Y+ \rho  \sqrt{\sigma  \tau } Y^2\right), \label{B1}\\
  \E((Y-\rho X)^2|X)&\leq&\frac{1-\rho ^2}{1+\sqrt{\sigma  \tau } \rho }\left(1+\eta  \sqrt{\rho } \sqrt[4]{\frac{\tau }{\sigma }} X+ \ \rho\sqrt{\sigma\tau} X^2\right).\label{B2}
 \end{eqnarray*}
 So
inequalities \eqref{C1} and \eqref{C2} hold  with $\delta=2\sqrt{\sigma\tau}$
for some constants $A,B\geq 0$.  The result follows from Lemma \ref{L2}.
\end{proof}
\section{Integrability conjecture}
Let $\delta>0$. Recall that  $(X_t)_{t\in T}$ is a quadratic harness on a non-empty open interval $T$ with parameters $(\eta,\theta,\sigma,\tau,\gamma)$ if it  satisfies \eqref{cov}, \eqref{mart} for all $s<t$ in $T$, and in addition that  \eqref{qmart1}  and \eqref{qmart2}  hold with equality, and  that  for all $s<t<u$ in $T$,
 \begin{equation}\label{LR}
 E(X_t|X_s,X_u)=\frac{u-t}{u-s}X_s+\frac{t-s}{u-s}X_t
 \end{equation}
and   \begin{multline}\label{Var}
 \Var(X_t|X_s,X_u)=F_{t,s,u}\Big(1+\theta \frac{X_u-X_s}{u-s}
+ \eta \frac{uX_s-sX_u}{u-s}
\\+\tau \frac{(X_u-X_s)^2}{(u-s)^2}+ \sigma
\frac{(uX_s-sX_u)^2}{(u-s)^2} -(1-\gamma)\frac{(uX_s-sX_u)(X_u-X_s)}{(u-s)^2}
 \Big),
\end{multline}
where
\begin{equation*}\label{Ftsu}
F_{t,s,u}=\frac{(u-t)(t-s)}{u(1+s\sigma)+\tau-s\gamma}.
\end{equation*}

Theorem \ref{T1}  does not use  \eqref{LR}, \eqref{Var}, so it does not use the all properties of a quadratic harness.
\begin{conjecture}\label{Con1}
Suppose that $(X_t)$ is a quadratic harness on $(1-\delta,1+\delta)$ for some $\delta>0$.
\begin{enumerate}
\item[(i)] If $0<\sigma\tau<1$ and  $1-2\sqrt{\sigma\tau}\leq \gamma\leq 1+2\sqrt{\sigma\tau}$, then  $\E(|X_t|^p)<\infty$ for all $0\leq p<2+\frac{1}{\sigma\tau}$.
\item[(ii)] If $\sigma\tau$ is small enough and $-1\leq \gamma\leq 1-2\sqrt{\sigma\tau}$, then $\E(|X_t|^p)<\infty$ for all $p\geq 0$.
\end{enumerate}
\end{conjecture}
By Theorem \ref{T1},
Conjecture \ref{Con1} is true for $\sigma\tau=0$, and this result has essentially been known, although it was stated only for  quadratic harnesses on $(0,\infty)$.
An even stronger version of Conjecture \ref{Con1}(ii), formulated by J. Weso\l owski, says that if
$-1\leq \gamma<1-2\sqrt{\sigma\tau}$ and $\sigma\tau$ is small enough then  $X_t$ is bounded. Here we indicate that this stronger version of Conjecture \ref{Con1}(ii) holds true for $\gamma=-1$.

\begin{proposition}\label{P2} Suppose that $(X_t)$ is a quadratic harness on $(1-\delta,1+\delta)$ for some $\delta>0$, with parameters $\gamma=-1$ and $\sigma\tau<1/921600\approx 10^{-6}$. Then for every $t\in(1-\delta,1+\delta)$, random variable $X_t$ has (at most) two values.
\end{proposition}
\begin{proof}[Sketch of proof] Since  $\sigma\tau<1/921600$, from Theorem \ref{T1} we deduce that the $4$-th moments exist.
From a longer calculation based on the arguments in the proof of \cite[Lemma 3.5]{Bryc-Matysiak-Wesolowski-04}, one can obtain the formulas for the first four moments, and compute
  the Hankel determinant
$$
\det \left[\begin{matrix}
1& 0 & t \\
0& t &\E(X_t^3)\\
t & \E(X_t^3) &\E (X_t^4)
\end{matrix}\right]
= (1+\gamma)\frac{ (t+\tau ) (1+t \sigma ) \left((\eta\tau+\theta )(\eta+\theta\sigma) +\left(1-\sigma
\tau\right)^2 \right)}{
   1-(2+\gamma) \sigma \tau}.
$$
Since $\gamma=-1$, the Hankel determinant is zero, and the distribution is concentrated on two points.
\end{proof}

 \subsection*{Acknowledgement} This research was partially supported by NSF
grant \#DMS-0904720. The author  thanks Politechnika Warszawska  for its hospitality during the work on the paper.
The research benefited from  several discussions with Jacek Weso\l owski.
\bibliographystyle{apa}
\bibliography{../vita,moms2011}

%
%
%

\end{document}